\documentclass[12pt,reqno]{amsart}
\usepackage{mathrsfs,amssymb,graphicx,verbatim,amsmath,bm}
\usepackage{paralist}
\usepackage[breaklinks,pdfstartview=FitH]{hyperref}
\renewcommand{\d}{\delta}

\newcommand{\e}{\varepsilon}
\newcommand{\R}{\mathbb R}
\newcommand{\B}{\mathbb{B}}
\newcommand{\1}{\mathbf 1}
\newcommand{\eqdef}{\stackrel{\mathrm{def}}{=}}

\renewcommand{\le}{\leqslant}
\renewcommand{\ge}{\geqslant}
\renewcommand{\leq}{\leqslant}

\newtheorem{theorem}{Theorem}[section]
\newtheorem{proposition}[theorem]{Proposition}
\newtheorem{lemma}[theorem]{Lemma}
\newtheorem{corollary}[theorem]{Corollary}

\theoremstyle{remark}
\newtheorem{conjecture}[theorem]{Conjecture}

\newtheorem{remark}[theorem]{Remark}
\newtheorem{question}[theorem]{Question}
\theoremstyle{definition}
\newtheorem{definition}[theorem]{Definition}

\renewcommand{\H}{\mathbb H}

\newcommand{\N}{\mathbb N}
\newcommand{\Z}{\mathbb Z}

\renewcommand{\setminus}{\smallsetminus}

\newcommand{\f}{\phi}

\begin{document}

\title[Vertical versus horizontal Poincar\'e inequalities]{Vertical versus horizontal Poincar\'e inequalities on the Heisenberg group}
\thanks{V.~L. was supported in part by ANR grant KInd. A.~N. was supported in part by NSF grant CCF-0832795, BSF grant 2010021, the Packard Foundation and the Simons Foundation. Part of this work was completed while A.~N. was visiting Universit\'e de Paris Est Marne-la-Vall\'ee.}

\author{Vincent Lafforgue}
\address{Laboratoire de Math\'ematiques - Analyse, Probabilit\'es, Mod\'elisation - Orl\'eans (MAPMO)
UMR CNRS 6628, Universit\'e d'Orl\'eans
Rue de Chartres, B.P. 6759 - 45067 Orl\'eans cedex 2} \email{vlafforg@math.jussieu.fr}

\author{Assaf Naor}
\address{Courant Institute of Mathematical Sciences, New York University, 251 Mercer Street, New York NY 10012, USA}
\email{naor@cims.nyu.edu}
%\date{\today}

\dedicatory{Dedicated to Joram Lindenstrauss}

\maketitle

\begin{abstract}
Let $\H= \left\langle a,b\,|\, a[a,b]=[a,b]a\ \wedge\ b[a,b]=[a,b]b\right\rangle$ be the discrete Heisenberg group, equipped with the left-invariant word metric $d_W(\cdot,\cdot)$ associated to the generating set $\{a,b,a^{-1},b^{-1}\}$. Letting $B_n= \{x\in \H:\ d_W(x,e_\H)\le n\}$ denote the corresponding closed ball of radius $n\in \N$, and writing  $c=[a,b]=aba^{-1}b^{-1}$, we prove that if $(X,\|\cdot\|_X)$ is a Banach space whose modulus of uniform convexity has power type $q\in [2,\infty)$ then there exists $K\in (0,\infty)$ such that every $f:\H\to X$ satisfies
\begin{multline*}
\sum_{k=1}^{n^2}\sum_{x\in B_n}\frac{ \|f(xc^k)-f(x)\|_X^q}{k^{1+q/2}}\\\le
K\sum_{x\in B_{21n}} \Big(\|f(xa)-f(x)\|^q_X+\|f(xb)-f(x)\|^q_X\Big).
\end{multline*}
It follows that for every $n\in \N$ the bi-Lipschitz distortion of every $f:B_n\to X$ is at least a constant multiple of $(\log n)^{1/q}$, an asymptotically optimal estimate as $n\to\infty$.
\end{abstract}

\section{Introduction}

The discrete Heisenberg group, denoted $\H$, is the group generated by two elements $a,b\in \H$, with the relations asserting that the commutator $[a,b]=aba^{-1}b^{-1}$ is in the center of $\H$. Thus $\H$ is given by the presentation $\H= \left\langle a,b\,|\, a[a,b]=[a,b]a\ \wedge\ b[a,b]=[a,b]b\right\rangle$. Write $c= [a,b]$ and let $e_\H$ denote the identity element of $\H$. The left-invariant word metric on $\H$ induced by the symmetric generating set $\{a,b,a^{-1},b^{-1}\}$ is denoted $d_W(\cdot,\cdot)$. For $n\in \N$ let $B_n= \{x\in \H:\ d_W(x,e_\H)\le n\}$ denote the corresponding closed ball of radius $n$.

A Banach space $(X,\|\cdot\|_X)$ is said to be uniformly convex if for every $\e\in (0,1)$ there exists $\d\in (0,1)$ such that every $x,y\in X$ with $\|x\|_X=\|y\|_X=1$ and $\|x-y\|_X\ge \e$ satisfy $\|x+y\|_X\le 2(1-\d)$. The supremum over those $\d\in (0,1)$ for which this holds true is denoted $\d_{(X,\|\cdot\|_X)}(\e)$, and is called the modulus of uniform convexity of $(X,\|\cdot\|_X)$.  An important theorem of Pisier~\cite{Pisier-martingales} asserts that every uniformly convex Banach space $(X,\|\cdot\|_X)$ admits an equivalent norm $\|\cdot\|$ for which there exist $q\in [2,\infty)$ and $\eta\in (0,1)$ such that $\d_{(X,\|\cdot\|)}(\e)\ge (\eta\e)^q$ for all $\e\in (0,1)$.  For concreteness we recall~\cite{Han56} that if $p\in (1,\infty)$ then $\ell_p$ satisfies such an estimate with $q=\max\{p,2\}$.

\begin{theorem}[Vertical versus horizontal Poincar\'e inequality]\label{thm:main}
For every $\eta\in (0,1)$ and $q\in [2,\infty)$ there exists $K=K(\eta,q)\in (0,\infty)$ with the following property. Suppose that $(X,\|\cdot\|_X)$ is a Banach space satisfying $\d_{(X,\|\cdot\|_X)}(\e)\ge (\eta\e)^q$ for every $\e\in (0,1)$. Then for every $n\in \N$ and every $f:\H\to X$ we have
\begin{multline}\label{eq:main}
\sum_{k=1}^{n^2}\sum_{x\in B_n}\frac{ \|f(xc^k)-f(x)\|_X^q}{k^{1+q/2}}\\\le
K\sum_{x\in B_{21n}} \Big(\|f(xa)-f(x)\|^q_X+\|f(xb)-f(x)\|^q_X\Big).
\end{multline}
\end{theorem}
The constant $21$ appearing in the range of the summation on the right hand side of~\eqref{eq:main} is an artifact of our proof and is not claimed to be sharp. The important point here is that the summation on the right hand side of~\eqref{eq:main} is over $x\in B_{\lambda n}$ for some universal constant $\lambda\in \N$. One can clearly make the same statement for word metrics induced by other finite symmetric generating sets of $\H$: the choice of generating set will only affect the value of $\lambda$.

An inspection of our proof of Theorem~\ref{thm:main} reveals that $K^{1/q}\lesssim 1/\eta$, but we will not explicitly track the value of such constants  in the ensuing discussion.  Here, and in what follows, we use $A \lesssim B$ and $B \gtrsim A$ to denote the estimate $A \leq CB$ for some absolute constant $C\in (0,\infty)$. If we need $C$ to depend on parameters, we indicate this by subscripts, thus e.g. $A \lesssim_\alpha B$ means that $A \leq C_\alpha B$ for some $C_\alpha\in (0,\infty)$ depending only on $\alpha$. We shall also use the notation $A\asymp B$ for $A\lesssim B\ \wedge\ B\lesssim A$, and similarly $A\asymp_\alpha B$ stands for  $A\lesssim_\alpha B\ \wedge\ B\lesssim_\alpha A$.

We call~\eqref{eq:main} a ``vertical versus horizontal Poincar\'e inequality" for the following reason. The right hand side of~\eqref{eq:main} is the $\ell_q$ norm of the discrete horizontal gradient of $f$: it measures the ``local" variation of $f$ along the edges of the Cayley graph of $\H$ (a.k.a. the horizontal edges in $\H$). The left hand side of~\eqref{eq:main} measures the ``global" variation of $f$ along the center of $\H$ (a.k.a. the vertical direction in $\H$). Theorem~\ref{thm:main} asserts that the global vertical variation of $f$ is always bounded by its local horizontal variation. Thus, if the right hand side of~\eqref{eq:main} is small then $f$ must collapse distances along the center of $\H$.

The (bi-Lipschitz) distortion of a finite metric space $(M,d_M)$ in a Banach space $(X,\|\cdot\|_X)$, denoted $c_X(M,d_M)\in [1,\infty)$, is the infimum over those $D\in [1,\infty)$ for which there exists an embedding $f:M\to X$ that satisfies $d_M(x,y)\le \|f(x)-f(y)\|_X\le D d_M(x,y)$ for all $x,y\in M$. When $X=\ell_p$ for some $p\in [1,\infty)$ it is customary to write $c_{\ell_p}(M,d_M)=c_p(M,d_M)$. The quantity $c_2(M,d_M)$ is known as the Euclidean distortion of $(M,d_M)$. Suppose that $(X,\|\cdot\|_X)$ satisfies the assumption of Theorem~\ref{thm:main} and that $f:\H\to X$ satisfies $d_{W}(x,y)\le \|f(x)-f(y)\|_X\le D d_W(x,y)$ for all $x,y\in B_{22n}$. Since $d_W(c^k,e_\H)\asymp \sqrt{k}$ for every $k\in \N$ and $|B_m|\asymp m^4$ for every $m\in \N$ (see e.g.~\cite{Bla03}), Theorem~\ref{thm:main} applied to $f$ yields the following estimate.
\begin{equation}\label{eq:distortion computation}
n^4\log n\lesssim \sum_{k=1}^{n^2} n^4\frac{k^{q/2}}{k^{1+q/2}}\lesssim_X n^4D^q.
\end{equation}
We therefore obtain the following corollary of Theorem~\ref{thm:main}.
\begin{corollary}[Sharp nonembeddabilty of  balls in $\H$]\label{cor:genX}
Fix $\eta\in (0,1)$ and $q\in [2,\infty)$. Suppose that $(X,\|\cdot\|_X)$ is a Banach space satisfying $\d_{(X,\|\cdot\|_X)}(\e)\ge (\eta\e)^q$ for every $\e\in (0,1)$. Then for every $n\in \N$ we have
$$
c_X\left(B_n,d_W\right)\gtrsim_{\eta} (\log n)^{1/q}.
$$
\end{corollary}
Corollary~\ref{cor:genX} yields an estimate on $c_X(B_n,d_W)$ in terms of the modulus of uniform convexity of $(X,\|\cdot\|_X)$ which is asymptotically best possible, up to constant factors that are independent of $n$. The following corollary states this explicitly for the case of special interest $X=\ell_p$, though one could equally well state such results for a variety of concrete spaces for which the modulus of uniform convexity has been computed (e.g., the same conclusion holds true with $\ell_p$ replaced by the Schatten class $S_p$, due to the computation of its modulus of uniform convexity in~\cite{Tom74}).

\begin{corollary}\label{cor:L_p} For every integer $n\ge 2$ we have
$$
p\in (1,2]\implies c_p(B_n,d_W)\asymp_p \sqrt{\log n},
$$
and
$$
p\in [2,\infty)\implies c_p(B_n,d_W)\asymp_p (\log n)^{1/p}.
$$
\end{corollary}
The lower bounds on $c_p(B_n,d_W)$ that appear in Corollary~\ref{cor:L_p} are a special case of Corollary~\ref{cor:genX} (in this case $q=\max\{p,2\}$). There are several ways to establish the asymptotically matching upper bounds. Assouad proved in~\cite{Ass83} that there exists $k\in \N$ and $1$-Lipschitz functions $\{\f_j:\H\to \R^k\}_{j=1}^\infty$ satisfying $\|\f_j(x)-\f_j(y)\|_\infty\gtrsim d_W(x,y)$ for every $x,y\in \H$ with $d_W(x,y)\in [2^{j-1},2^j]$: by concatenating $\f_1,\ldots,\f_m$ for $m\asymp \log n$ one sees that the distortion lower bounds of Corollary~\ref{cor:L_p} are indeed asymptotically sharp. An alternative  embedding of $(B_n,d_W)$ into $\ell_p$  with the desired distortion bound was found by Tessera in~\cite{Tes08}. For $p=2$ one can use the explicit closed-form embedding of $\H$ into $\ell_2$ of~\cite{LN06}, where for every $\e\in (0,1)$ a mapping $f_\e:\H\to \ell_2$ is given with $d_W(x,y)^{1-\e}\le \|f_\e(x)-f_\e(y)\|_2\lesssim d_W(x,y)^{1-\e}/\sqrt{\e}$ for every $x,y\in \H$. Setting $\e=1/\log n$ shows that $c_2(B_n,d_W)\lesssim \sqrt{\log n}$.

The above distortion bounds complete a sequence of investigations of the (non)embeddability of the Heisenberg group into ``nice" Banach spaces. A famous observation of Semmes~\cite{Sem96} shows that Pansu's differentiation theorem for Carnot groups~\cite{Pan89} implies that $\H$ does not admit a bi-Lipschitz embedding into $\R^n$ for any $n\in \N$. Alternative proofs of this fact were obtained by Cheeger~\cite{Che99} and Pauls~\cite{Pau01}. The fact that the Pansu-Semmes argument can be extended to certain infinite dimensional targets, yielding in particular the bi-Lipschitz nonembeddability of $\H$ into any uniformly convex Banach space, was obtained independently  by~\cite{CK06} and~\cite{LN06} (via different arguments). In~\cite{CK10} Cheeger and Kleiner proved that $\H$ does not admit a bi-Lipschitz embedding into any $L_1(\mu)$ space, a result that is important for an application to theoretical computer science that will be mentioned in Section~\ref{sec:conj}. Note that not all uniformly convex Banach spaces admit a bi-Lipschitz embedding into an $L_1(\mu)$ space (e.g. $\ell_p$ for $p\in (2,\infty)$), but the Cheeger-Kleiner theorem does yield a new proof of the nonembeddability of $\H$ into some uniformly convex spaces of interest, such as $L_p$ for $p\in (1,2]$, because they are isomorphic to subspaces of $L_1$. In~\cite{CK10-mon} Cheeger and Kleiner discovered a different proof of the nonembeddability of $\H$ into an $L_1(\mu)$ space.

The results quoted above imply that $\lim_{n\to\infty} c_X(B_n,d_W)=\infty$ for the respective target Banach spaces $X$, but they give no information on the rate at which $c_X(B_n,d_W)$ tends to $\infty$ with $n$. In order to obtain such quantitative nonembeddability results one needs to overcome additional (conceptual and technical) issues. The first progress in this direction was due to~\cite{ckn}, where it is shown that $c_1(B_n,d_W)\gtrsim (\log n)^\kappa$ for some universal constant $\kappa>0$. In~\cite{ANT10} it is shown that $c_2(B_n,d_W)\gtrsim \sqrt{\log n}$ and that if $(X,\|\cdot\|_X)$ is a Banach space satisfying $\d_{(X,\|\cdot\|_X)}(\e)\gtrsim_X \e^q$ for every $\e\in (0,1)$ then $c_X(B_n,d_W)\gtrsim_X (\log n/\log\log n)^{1/q}$. Recently, Li~\cite{Li12} obtained a quantitative version of Pansu's differentiation theorem which yields the estimate $c_X(B_n,d_W)\gtrsim_X (\log n)^{\theta_X}$ for some $\theta_X>0$.

As explained above, except for the asymptotic evaluation of the Euclidean distortion $c_2(B_n,d_W)$ that was obtained in~\cite{ANT10}, computing $c_X(B_n,d_W)$ up to constant factors that are independent of $n$ remained open for all non-Hilbertian Banach spaces (the lower bound in~\cite{ANT10} was off by an iterated logarithm factor). Corollary~\ref{cor:genX} resolves this problem for uniformly convex Banach spaces. Other than yielding a complete result for an important class of Banach spaces, the significance of Corollary~\ref{cor:genX} is that its proof is different from the approaches that have been used thus far in the literature.

Specifically, all the above mentioned results first use a limiting argument that shows that it suffices to rule out low-distortion embeddings that belong to a certain ``structured" subclass of all the possible embeddings (e.g. in the case of the Pansu-Semmes proof one argues that it suffices to deal with group homomorphisms). The proofs in~\cite{Pan89,Sem96,Che99,Pau01,CK06,LN06,CK10,CK10-mon,ckn,Li12} all apply this general ``metric differentiation" strategy. The proof in~\cite{ANT10} uses a different but related approach: one first argues that it suffices to rule out embeddings that are $1$-cocycles with respect to an isometric action of $\H$ on $X$; when $X$ is Hilbert space this is done via an argument of Aharoni, Maurey and Mityagin~\cite{AMM85} and Gromov~\cite{CTV07}, and for general uniformly convex $X$ this is done via an argument of~\cite{NP11}. When $X$ is a general uniformly convex Banach space one proceeds in~\cite{ANT10} via an algebraic argument and a quantitative version of a mean ergodic theorem, yielding a bound that is off by an iterated logarithm factor. This argument fails to yield a Poincar\'e inequality such as~\eqref{eq:main}. The proof of the sharp estimate on the Euclidean distortion $c_2(B_n,d_W)$ proceeds in~\cite{ANT10} by proving the Hilbertian case of the Poincar\'e inequality~\eqref{eq:main}: this is done using a theorem of Guichardet~\cite{Gui72} that further reduces the problem to co-boundaries, and since we are interested in an inequality that involves squares of Euclidean distances, one can use the available orthogonality to reduce to the case where the underlying unitary representation is irreducible. Co-boundaries with respect to irreducible representations may then be treated separately via a direct argument.

We do not know how to prove Theorem~\ref{thm:main} using the above strategies: metric differentiation arguments seem to inherently lose (at least) an iterated logarithm factor, and in the only case where a sharp bound was proved the argument heavily uses the Hilbertian structure. Our approach is therefore entirely different: we prove the inequality~\eqref{eq:main} directly via an analytic argument that relies on generalized Littlewood-Paley $g$-function estimates.  While our method does not yield an improved lower bound on $c_1(B_n,d_W)$, it suggests a clean isoperimetric-type inequality that, if true, would yield the (at present still conjectural) sharp estimate $c_1(B_n,d_W)\gtrsim \sqrt{\log n}$. This conjecture, whose investigation is deferred to future work, is discussed in Section~\ref{sec:conj}.

\section{Inequalities on the real Heisenberg group}\label{sec:continuous heisenberg}
We start by setting some (mostly standard) notation and terminology. The Heisenberg group $\H$ can be identified with the following matrix group, equipped with matrix multiplication.
$$
\H=\left\{\left(\begin{array}{cccccc}
1 & x & z \\
0 &  1 & y\\
0 & 0 & 1
\end{array}\right):\  x,y,z\in \Z\right\}.
$$
Under this identification, we have
$$
a= \left(\begin{array}{cccccc}
1 & 1 & 0 \\
0 &  1 & 0\\
0 & 0 & 1
\end{array}\right) \qquad \mathrm{and}\qquad b=\left(\begin{array}{cccccc}
1 & 0 & 0 \\
0 &  1 & 1\\
0 & 0 & 1
\end{array}\right).
$$
Thus
$$
c=aba^{-1}b^{-1}=\left(\begin{array}{cccccc}
1 & 0 & 1 \\
0 &  1 & 0\\
0 & 0 & 1
\end{array}\right).
$$
We will reason below about the real Heisenberg group, denoted $\H(\R)$, which is defined  as
$$
\H(\R)\eqdef \left\{\left(\begin{array}{cccccc}
1 & x & z \\
0 &  1 & y\\
0 & 0 & 1
\end{array}\right):\  x,y,z\in \R\right\}.
$$
We will use the following notation for every $x,y,z\in \R$.
$$
a^x\eqdef \left(\begin{array}{cccccc}
1 & x & 0 \\
0 &  1 & 0\\
0 & 0 & 1
\end{array}\right), \quad  b^y\eqdef \left(\begin{array}{cccccc}
1 & 0 & 0 \\
0 &  1 & y\\
0 & 0 & 1
\end{array}\right),\quad c^z\eqdef \left(\begin{array}{cccccc}
1 & 0 & z \\
0 &  1 & 0\\
0 & 0 & 1
\end{array}\right).
$$
Thus,
$$
c^zb^ya^x= \left(\begin{array}{cccccc}
1 & x & z \\
0 &  1 & y\\
0 & 0 & 1
\end{array}\right).
$$

It is convenient to identify $\H(\R)$ with $\R^3$. In particular, for a Banach space $(X,\|\cdot\|_X)$ and a mapping $f:\H(\R)\to X$, we identify $f(x,y,z)$ with $f(c^zb^ya^x)$. Under this identification, the Lebesgue measure on $\R^3$ is a Haar measure on $\H(\R)$; below we denote this measure on $\H(\R)$ by $\mu$. The spaces $\R,\R^2,\H(\R)$ will always be understood to be endowed with the Lebesgue measure. Thus for $p\in [1,\infty)$ the Lebesgue-Bochner spaces $L_p(\R,X), L_p(\R^2,X), L_p(\H(\R),X)$ are defined unambiguously. For $f\in L_p(\H(\R),X)$ define $f^c:\R\to L_p(\R^2,X)$ by
\begin{equation}\label{eq:def fc}
f^c(z)(x,y)=f(x,y,z).
\end{equation}
Given $\psi\in L_1(\R)$, we define the convolution $\psi*f\in L_p(\H(\R),X)$ by $\psi*f(x,y,z)=(\psi*f^c)(z)(x,y)$, i.e.,
\begin{equation}\label{eq:def convolution}
\psi*f(x,y,z)\eqdef\int_\R \psi(u)f(x,y,z-u)du\in X.
\end{equation}
Equivalently, $\psi*f$ is the usual group convolution  of $f$ with the measure supported on the center of $\H(\R)$ whose density is $\psi$.

Suppose that $f:\H(\R)\to X$ is smooth. The identification of $\H(\R)$ with $\R^3$ gives meaning to the partial derivatives $\frac{\partial f}{\partial x}, \frac{\partial f}{\partial y}$. We define the left-invariant vector fields $\partial_af,\partial_bf:\H(\R)\to X$ by
\begin{equation}\label{eq:def partia a}
\partial_a f(x,y,z)\eqdef \frac{\partial f}{\partial x}(x,y,z),
\end{equation}
and
\begin{equation}\label{eq:def partia b}
\partial_b f(x,y,z)\eqdef \frac{\partial f}{\partial y}(x,y,z)+x\frac{\partial f}{\partial z}(x,y,z).
\end{equation}
The horizontal gradient of $f$ is then defined as
\begin{equation}\label{eq:def horizontal grad}
\nabla_\H f\eqdef(\partial_af,\partial_bf):\H(\R)\to X\oplus X.
\end{equation}
Thus for $p\in [1,\infty)$ and $x,y,z\in \R$ we have
\begin{multline*}
\left\|\nabla_\H f(x,y,z)\right\|_{\ell_p^2(X)}\\=\left(\left\|\frac{\partial f}{\partial x}(x,y,z)\right\|_X^p+\left\|\frac{\partial f}{\partial y}(x,y,z)+x\frac{\partial f}{\partial z}(x,y,z)\right\|_X^p\right)^{1/p}.
\end{multline*}

Theorem~\ref{thm:real} below, the case $p=q$ of which establishes a continuous version of the Poincar\'e inequality of Theorem~\ref{thm:main}, is the main result of this section. Theorem~\ref{thm:main} itself will be shown in Section~\ref{sec:proof main} to follow from Theorem~\ref{thm:real} via a partition of unity argument.

\begin{theorem}[Real vertical versus horizontal Poincar\'e inequality]\label{thm:real}
Suppose that $q\in [2,\infty)$ and $p\in (1,q]$. Let $(X,\|\cdot\|_X)$ be a Banach space satisfying $\d_{(X,\|\cdot\|_X)}(\e)\ge (\eta\e)^q$ for every $\e\in (0,1)$ and some $\eta\in (0,1)$. Then every smooth and compactly supported $f:\H(\R) \to X$ satisfies
\begin{multline}\label{eq:continuous main}
\left(\int_0^\infty \left(\int_{\H(\R)}\|f(hc^t)-f(h)\|_X^pd\mu(h)\right)^{q/p}\frac{dt}{t^{1+q/2}}\right)^{1/q}\\\lesssim_{\eta,p,q} \left(\int_{\H(\R)} \left\|\nabla_\H f(h)\right\|_{\ell_p^2(X)}^p d\mu(h)\right)^{1/p}.
\end{multline}
\end{theorem}

In preparation for the proof of Theorem~\ref{thm:real}, we first prove some preparatory lemmas. The Poisson kernel on $\R$ is defined as usual for every $t,x\in \R$ by
$$
P_t(x)\eqdef \frac{t}{\pi(t^2+x^2)}.
$$
We also write
\begin{equation}\label{eq:defQ}
Q_t(x)\eqdef \frac{\partial}{\partial t} P_t(x)=\frac{x^2-t^2}{\pi(t^2+x^2)^2},
\end{equation}
and
\begin{equation}\label{eq:defR}
 R_t(x)\eqdef \frac{\partial}{\partial x} P_t(x)=\frac{-2tx}{\pi(t^2+x^2)^2}.
\end{equation}
We record for future use that for every $t\in (0,\infty)$,
\begin{equation}\label{eq:P_t 1/2 moment}
\int_0^\infty \sqrt{x}P_t(x)dx= \frac{\sqrt{t}}{\pi}\int_0^\infty \frac{\sqrt{y}}{1+y^2}dy=\sqrt{\frac{t}{2}},
\end{equation}
\begin{equation}\label{eq:norm Q}
\|Q_t\|_{L_1(\R)}=\frac{1}{\pi t} \int_\R \frac{|1-y^2|}{(1+y^2)^2}dy =\frac{2}{\pi t},
\end{equation}
and
\begin{equation}\label{eq:norm R}
\|R_t\|_{L_1(\R)}=\frac{2}{\pi t} \int_\R \frac{|y|}{(1+y^2)^2}dy =\frac{2}{\pi t}.
\end{equation}
(The precise constants appearing in~\eqref{eq:P_t 1/2 moment}, \eqref{eq:norm Q} and \eqref{eq:norm R} are not needed in the ensuing discussion: only the stated dependence on $t$ up to constant factors will be used.)

The main step of the proof of Theorem~\ref{thm:real} is Proposition~\ref{prop:main} below.
\begin{proposition}\label{prop:main}
Fix $p,q\in [1,\infty)$ and let $(X,\|\cdot\|_X)$ be a Banach space. Every smooth and compactly supported $f:\H(\R)\to X$ satisfies
\begin{multline*}
\left(\int_0^\infty \left(\int_{\H(\R)}\|f(hc^t)-f(h)\|_X^pd\mu(h)\right)^{q/p}\frac{dt}{t^{1+q/2}}\right)^{1/q}\\\lesssim
\left(\int_0^\infty t^{q-1}\left\|Q_t*\nabla_\H f\right\|_{L_p(\H(\R),\ell_p^2(X))}^qdt\right)^{1/q}.
\end{multline*}
\end{proposition}

\begin{remark}
Continuing with the notation of Proposition~\ref{prop:main}, for every $t\in (0,\infty)$ denote $\tau_tf:\H(\R)\to X$ by $\tau_tf(h)=f(hc^t)$. Write also $S_t=\frac{\pi}{2}tQ_t$. By~\eqref{eq:norm Q} this yields the normalization $\|S_t\|_{L_1(\R)}=1$. The assertion of Proposition~\ref{prop:main} can then be rewritten as
\begin{equation*}
\left\|\left\|\frac{\tau_tf-f}{\sqrt{t}}\right\|_{L_p(\H(\R),X)}\right\|_{L_q\left(\R_+,\frac{dt}{t}\right)}\lesssim \left\|\left\|S_t*\nabla_\H f\right\|_{L_p(\H(\R),\ell_p^2(X))}\right\|_{L_q\left(\R_+,\frac{dt}{t}\right)}.
\end{equation*}
Introducing $S_t$ in this way helps explain the meaning of the powers of $t$ that appear in the statement of Proposition~\ref{prop:main}, but it makes the ensuing proofs more cumbersome. We will therefore continue working with $Q_t$ rather than $S_t$ in what follows.
\end{remark}

Assuming the validity of Proposition~\ref{prop:main} for the moment, we now proceed to show how it implies Theorem~\ref{thm:real}.
\begin{proof}[Proof of Theorem~\ref{thm:real}]
Given a Banach space space $(\B,\|\cdot\|_\B)$, for every function $\phi\in L_p(\R,\B)$ its generalized Hardy-Littlewood $g$-function $\mathfrak{G}_q(\phi):\R\to [0,\infty]$ is defined as follows.
\begin{equation}\label{eq:def g function}
\mathfrak{G}_q(\phi)(x)\eqdef \left(\int_0^\infty t^{q-1}\left\|Q_t*\phi(x)\right\|_\B^qdt\right)^{1/q}.
\end{equation}
A beautiful theorem of Martinez, Torrea and Xu~\cite[Thm.~2.1]{MTX06} asserts that if $\d_{(\B,\|\cdot\|_\B)}(\e)\ge (\xi\e)^q$ for every $\e\in (0,1)$ and some constant $\xi\in (0,1)$ then for every $p\in (1,\infty)$,
\begin{equation}\label{eq:MTX}
\phi\in L_p(\R,\B)\implies\left\|\mathfrak{G}_q(\phi)\right\|_{L_p(\R,\B)}\lesssim_{\xi,p,q} \|\phi\|_{L_p(\R,\B)}.
\end{equation}
In fact, it is shown in~\cite{MTX06} that the validity of~\eqref{eq:MTX} for some $p\in (1,\infty)$ (equivalently for all $p\in (1,\infty)$) is equivalent to $\B$ admitting an equivalent norm whose modulus of uniform convexity is at least a constant multiple of $\e^q$. Note that Theorem~2.1 of~\cite{MTX06} asserts~\eqref{eq:MTX} under the assumption that $(\B,\|\cdot\|_\B)$ has martingale cotype $q$, but this follows from our assumption on the modulus of uniform convexity of $(\B,\|\cdot\|_\B)$ by important work of Pisier~\cite{Pisier-martingales}.

We shall apply~\eqref{eq:MTX} to $\B=L_p(\R^2,\ell_p^2(X))$. Since we are assuming in Theorem~\ref{thm:real} that $p\in (1,q]$, by a result of Figiel~\cite{Fig76} there exists $\xi=\xi(\eta,p,q)\in (0,1)$ such that  $\d_{(\B,\|\cdot\|_\B)}(\e)\ge (\xi\e)^q$ for every $\e\in (0,1)$ (see Corollary~6.4 in~\cite{MN-towards} for an explicit dependence of $\xi$ on $\eta,p,q$). Recalling~\eqref{eq:def fc}, consider the function $\phi=(\nabla_\H f)^c:\R\to \B$. Then,
\begin{align}
&\nonumber\left(\int_0^\infty t^{q-1}\left\|Q_t*\nabla_\H f\right\|_{L_p(\H(\R),\ell_p^2(X))}^qdt\right)^{1/q}\\
\label{eq:use definitions B phi}
&=\left(\int_0^\infty \left(\int_\R \left(t\left\|Q_t*\phi(z)\right\|_{\B}\right)^pdz\right)^{q/p} \frac{dt}{t}\right)^{1/q}\\
\label{eq:use ple q}&\le \left(\int_\R\left(\int_0^\infty \left(t\left\|Q_t*\phi(z)\right\|_{\B}\right)^q \frac{dt}{t}\right)^{p/q}dz\right)^{1/p}\\
&=\label{eq:use MTX} \left\|\mathfrak{G}_q(\phi)\right\|_{L_p(\R,\B)}\lesssim_{\eta,p,q}  \|\phi\|_{L_p(\R,\B)},
\end{align}
where in~\eqref{eq:use definitions B phi} we used the definitions of $\B$ and $\phi$, in~\eqref{eq:use ple q} we used the fact that $p\le q$, and in~\eqref{eq:use MTX} we used~\eqref{eq:def g function} and~\eqref{eq:MTX}. Noting that by the definition of $\phi$ we have $\|\phi\|_{L_p(\R,\B)}=\|\nabla_\H f\|_{L_p(\H(\R),\ell_p^2(X))}$, the desired inequality~\eqref{eq:continuous main} follows from Proposition~\ref{prop:main}.
\end{proof}

We now pass to the proof of Proposition~\ref{prop:main}.
\begin{lemma}\label{lem:increment} Suppose that $p\in [1,\infty)$ and $t\in (0,\infty)$. Then for every Banach space $(X,\|\cdot\|_X)$ and every smooth and compactly supported $f:\H(\R)\to X$ we have
\begin{equation}\label{eq:increment}
\left\|Q_t*f-Q_{2t}*f \right\|_{L_p(\H(\R),X)}\lesssim \sqrt{t}\cdot \left\|Q_t*\nabla_\H f\right\|_{L_p(\H(\R),\ell_p^2(X))}.
\end{equation}
\end{lemma}
\begin{proof}
The semigroup property of $P_t$ implies that $P_{2t}=P_t*P_t$. By differentiating this identity with respect to $t$ we have $Q_{2t}=P_t*Q_t$. Consequently, for every $h\in \H(\R)$ we have
\begin{multline}\label{eq:use semigroup of root}
Q_t*f(h)-Q_{2t}*f(h)=Q_t*f(h)-P_t*Q_t*f(h)\\=\int_\R P_t(u)\left(Q_t*f(h)-Q_t*f(hc^{-u})\right)du,
\end{multline}
where we used the fact that $\int_\R P_t(u)du=1$.

For every $s\in [0,\infty)$ let $\gamma_s:[0,4\sqrt{s}]\to \H(\R)$ be the commutator path joining $e_\H$ and $c^{s}$, i.e.,

$$
\gamma_s(\theta)\eqdef\left\{\begin{array}{ll}
a^{\theta}& \mathrm{if}\ 0\le \theta\le \sqrt{s},\\
a^{\sqrt{s}} b^{\theta-\sqrt{s}} & \mathrm{if}\ \sqrt{s}\le \theta \le 2\sqrt{s},\\
a^{\sqrt{s}} b^{\sqrt{s}}a^{-\theta+2\sqrt{s}} & \mathrm{if}\ 2\sqrt{s}\le \theta\le 3\sqrt{s},\\
a^{\sqrt{s}} b^{\sqrt{s}}a^{-\sqrt{s}} b^{-\theta+3\sqrt{s}} & \mathrm{if}\ 3\sqrt{s}\le \theta\le 4\sqrt{s}.
\end{array}\right.
$$
For every $u\in [0,\infty)$ and $h\in \H(\R)$ we then have
\begin{multline}\label{eq:derive along commutator positive}
Q_t*f(h)-Q_t*f(hc^{-u})=Q_t*f(hc^{-u}\gamma_u(4\sqrt{u}))-Q_t*f(hc^{-u})\\
=\int_0^{4\sqrt{u}} \frac{d}{d\theta}Q_t*f(hc^{-u}\gamma_u(\theta)) d\theta.
\end{multline}
Recalling~\eqref{eq:def partia a} and~\eqref{eq:def partia b}, observe that for every $\theta\in [0,4\sqrt{u}]$ we have
\begin{multline}\label{eq;vector fields appear 1}
 \frac{d}{d\theta}Q_t*f(hc^{-u}\gamma_u(\theta))\in \left\{\partial_a Q_t*f(hc^{-u}\gamma_u(\theta)), \partial_b Q_t*f(hc^{-u}\gamma_u(\theta))\right\}\\
 =\left\{Q_t*\partial_a f(hc^{-u}\gamma_u(\theta)), Q_t*\partial_b f(hc^{-u}\gamma_u(\theta))\right\},
\end{multline}
where we also used the fact that convolution with $Q_t$ commutes with $\partial_a$ and $\partial_b$. Recalling~\eqref{eq:def horizontal grad}, we deduce from~\eqref{eq:derive along commutator positive} and~\eqref{eq;vector fields appear 1} that
\begin{align}
&\nonumber\left(\int_{\H(\R)}\left\|\int_0^\infty P_t(u)\left(Q_t*f(h)-Q_t*f(hc^{-u})\right)du\right\|_X^pd\mu(h)\right)^{1/p}\\
\nonumber&\le \int_0^\infty \int_0^{4\sqrt{u}} P_t(u)\left(\int_{\H(\R)}\left\|Q_t*\nabla_\H f(hc^{-u}\gamma_u(\theta))\right\|_{\ell_p^2(X)}^pd\mu(h)\right)^{1/p}d\theta du\\ \label{eq:right invariance}
&=\left(\int_0^\infty 4\sqrt{u}P_t(u)du\right)\|Q_t*\nabla_\H f\|_{L_p(\H(\R),\ell_p^2(X))}\\
&=\label{eq:use 1/2 integral}2\sqrt{2t}\cdot \|Q_t*\nabla_\H f\|_{L_p(\H(\R),\ell_p^2(X))},
\end{align}
where in~\eqref{eq:right invariance} we used the right-invariance of the Haar measure $\mu$ on $\H(\R)$ and in~\eqref{eq:use 1/2 integral} we used~\eqref{eq:P_t 1/2 moment}.

The analogue of~\eqref{eq:derive along commutator positive} for $u\in (-\infty,0]$ is the identity
\begin{equation*}
Q_t*f(h)-Q_t*f(hc^{-u})=
-\int_0^{4\sqrt{|u|}} \frac{d}{d\theta}Q_t*f(h\gamma_{-u}(\theta)) d\theta.
\end{equation*}
This, combined with the above reasoning yields the estimate
\begin{multline}\label{eq:negative half}
\left(\int_{\H(\R)}\left\|\int_{-\infty}^0 P_t(u)\left(Q_t*f(h)-Q_t*f(hc^{-u})\right)du\right\|_X^pd\mu(h)\right)^{1/p}\\
\le 2\sqrt{2t}\cdot \|Q_t*\nabla_\H f\|_{L_p(\H(\R),\ell_p^2(X))}.
\end{multline}
\eqref{eq:use 1/2 integral} and~\eqref{eq:negative half} combined with~\eqref{eq:use semigroup of root} yields the desired inequality~\eqref{eq:increment}.
\end{proof}

\begin{lemma}\label{lem:integrated increment}
Fix $p,q\in [1,\infty)$. For every Banach space $(X,\|\cdot\|_X)$, every smooth and compactly supported $f:\H(\R)\to X$ satisfies
\begin{multline*}
\left(\int_0^\infty t^{\frac{q}{2}-1}\left\|Q_t*f\right\|_{L_p(\H(\R),X)}^qdt\right)^{1/q}\\\lesssim \left(\int_0^\infty t^{q-1}\left\|Q_t*\nabla_\H f\right\|_{L_p(\H(\R),\ell_p^2(X))}^qdt\right)^{1/p}.
\end{multline*}
\end{lemma}
\begin{proof}
Observe that $\lim_{t\to\infty} Q_t*f=0$ in $L_p(\H(\R),X)$. Indeed, by Young's inequality,
$$
\|Q_t*f\|_{L_p(\H(\R),X)}\le \|Q_t\|_{L_1(\R)}\cdot \|f\|_{L_p(\H(\R),X)}\stackrel{\eqref{eq:norm Q}}{=}\frac{2}{\pi t}\cdot \|f\|_{L_p(\H(\R),X)}.
$$
We therefore have the identity $$Q_t*f=\sum_{m=1}^\infty \left(Q_{2^{m-1}t}*f-Q_{2^mt}*f\right),$$
from which it follows that
\begin{align}
\nonumber&\left(\int_0^\infty t^{\frac{q}{2}-1}\left\|Q_t*f\right\|_{L_p(\H(\R),X)}^qdt\right)^{1/q}\\
\nonumber &\le \sum_{m=1}^\infty \left(\int_0^\infty t^{\frac{q}{2}-1}\left\|Q_{2^{m-1}t}*f-Q_{2^mt}*f\right\|_{L_p(\H(\R),X)}^qdt\right)^{1/q}\\
\label{eq:use increment lemma}&\lesssim \sum_{m=1}^\infty \left(\int_0^\infty t^{\frac{q}{2}-1}(2^{m-1}t)^{q/2}\left\|Q_{2^{m-1}t}*\nabla_\H f\right\|_{L_p(\H(\R),\ell_p^2(X))}^qdt\right)^{1/q}\\
\label{eq:change of variable summand}&= \left(\sum_{m=1}^\infty \frac{1}{2^{(m-1)/2}}\right)\left(\int_0^\infty s^{q-1}\left\|Q_s*\nabla_\H f\right\|_{L_p(\H(\R),\ell_p^2(X))}^qds\right)^{1/p},
\end{align}
where~\eqref{eq:use increment lemma} uses Lemma~\ref{lem:increment}, and~\eqref{eq:change of variable summand} follows from the change of variable $s=2^{m-1}t$ in each of the summands of~\eqref{eq:use increment lemma}.
\end{proof}

Lemma~\ref{lem:enter g function} below is an analogue of a standard fact in Littlewood-Paley theory: it is commonly stated for real-valued functions defined on $\R$, while we need a statement for Banach space-valued functions defined on $\H(\R)$. In the real-valued case this fact has several known proofs, and while we do not know how to extend all of them to the vector-valued setting, the argument below is nothing more than the obvious modification of the corresponding proof in Stein's book~\cite{Ste70}.

\begin{lemma}\label{lem:enter g function}
Fix $p,q\in [1,\infty)$ and let $(X,\|\cdot\|_X)$ be a Banach space. Every smooth and compactly supported $f:\H(\R)\to X$ satisfies
\begin{multline}\label{eq:g function besov}
\left(\int_0^\infty \left(\int_{\H(\R)}\|f(hc^t)-f(h)\|_X^pd\mu(h)\right)^{q/p}\frac{dt}{t^{1+q/2}}\right)^{1/q}\\\lesssim
\left(\int_0^\infty t^{\frac{q}{2}-1}\left\|Q_t*f\right\|_{L_p(\H(\R),X)}^qdt\right)^{1/q}.
\end{multline}
\end{lemma}

\begin{proof}
For every $(h,t)\in \H(\R)\times (0,\infty)$ write
\begin{multline}\label{eq:three term semigroup}
f(hc^t)-f(x)=[f(hc^t)-P_t*f(hc^t)]\\+[P_t*f(hc^t)-P_t*f(h)]+[P_t*f(h)-f(h)].
\end{multline}
We proceed to bound each of the terms in the right hand side of~\eqref{eq:three term semigroup} separately. Firstly,
\begin{align}
\nonumber&\left(\int_0^\infty \left(\int_{\H(\R)}\|f(hc^t)-P_t*f(hc^t)\|_X^pd\mu(h)\right)^{q/p}\frac{dt}{t^{1+q/2}}\right)^{1/q}
\\\label{eq:recall Q_t}&=  \left(\int_0^\infty \left(\int_{\H(\R)}\left\|\int_0^t Q_s*f(hc^t)ds\right\|_{X}^pd\mu(h)\right)^{q/p} \frac{dt}{t^{1+q/2}}\right)^{1/q}
\\\label{eq:use triangle in Lp}&\le \left(\int_0^\infty \left(\int_0^t \left\|Q_s*f\right\|_{L_p(\H(\R),X)} ds\right)^q\frac{dt}{t^{1+q/2}}\right)^{1/q}
\\
\label{eq:first of three}&\le 2\left(\int_0^\infty t^{\frac{q}{2}-1}\|Q_t*f(x)\|_{L_p(\H(\R),X)}^qdt\right)^{1/q},
\end{align}
where for~\eqref{eq:recall Q_t} recall that in~\eqref{eq:defQ} we defined $Q_t$ as the time derivative of $P_t$, in~\eqref{eq:use triangle in Lp} we used the triangle inequality in $L_p(\H(\R),X)$ and the fact that the Haar measure $\mu$ is right-invariant,  and in~\eqref{eq:first of three} we used Hardy's inequality (see~\cite[Sec.~A.4]{Ste70}). The rightmost term in the right hand side of~\eqref{eq:three term semigroup} is bounded using the same argument, yielding the following inequality.
\begin{multline}\label{eq:third of three}
\left(\int_0^\infty \left(\int_{\H(\R)}\|P_t*f(h)-f(h)\|_X^pd\mu(h)\right)^{q/p}\frac{dt}{t^{1+q/2}}\right)^{1/q}\\
\le 2\left(\int_0^\infty t^{\frac{q}{2}-1}\|Q_t*f(x)\|_{L_p(\H(\R),X)}^qdt\right)^{1/q}.
\end{multline}

To bound the middle term in the right hand side of~\eqref{eq:three term semigroup} recall that by the semigroup property of the Poisson kernel, for every $t\in (0,\infty)$ we have $P_{t}=P_{t/2}*P_{t/2}$. As we have done in the proof of Lemma~\ref{lem:increment}, differentiation of this identity with respect to $t$ yields $Q_{t}=P_{t/2}*Q_{t/2}$. Consequently, recalling the definition of $R_t$ in~\eqref{eq:defR},
\begin{equation}\label{eq:R derivative}
\frac{\partial}{\partial t} R_t=\frac{\partial}{\partial x} Q_t= \frac{\partial}{\partial x}\left(P_{t/2}*Q_{t/2}\right)=R_{t/2}*Q_{t/2}.
\end{equation}
Observe that $\lim_{t\to \infty} R_t*f=0$ in $L_p(\H(\R),X)$. Indeed,
\begin{equation*}
\|R_t*f\|_{L_p(\H(\R),X)}\le \|R_t\|_{L_1(\R)}\cdot \|f\|_{L_p(\H(\R),X)}\stackrel{\eqref{eq:norm R}}{=}\frac{2}{\pi t}\cdot \|f\|_{L_p(\H(\R),X)}.
\end{equation*}
Consequently,
\begin{equation}\label{R integral representation}
R_t*f=-\int_t^\infty \frac{\partial}{\partial s} (R_s*f) ds\stackrel{\eqref{eq:R derivative}}{=}-\int_t^\infty R_{s/2}*(Q_{s/2}*f)ds.
\end{equation}
The middle term in the right hand side of~\eqref{eq:three term semigroup} can therefore be rewritten as follows.
\begin{multline}\label{eq:middle identity}
P_t*f(hc^t)-P_t*f(h)=\int_0^t R_t*f(hc^u)du\\=-\int_0^t \int_t^\infty R_{s/2}*(Q_{s/2}*f)(hc^u)dsdu.
\end{multline}
Observe that for every $t\in (0,\infty)$ we have
\begin{align}
&\nonumber\left(\int_{\H(\R)}\left\|\int_0^t \int_t^\infty R_{s/2}*(Q_{s/2}*f)(hc^u)dsdu\right\|_X^pd\mu(h)\right)^{1/p}\\
&\le \label{eq:triangle for norm} \int_0^t \int_t^\infty \left(\int_{\H(\R)}\left\|R_{s/2}*(Q_{s/2}*f)(hc^u)\right\|_X^pd\mu(h)\right)^{1/p}ds du
\\&= t\int_t^\infty \left\|R_{s/2}*(Q_{s/2}*f)\right\|_{L_p(\H(\R),X)}ds,\label{eq:t factor}
\end{align}
where~\eqref{eq:triangle for norm} uses the triangle inequality in $L_p(\H(\R),X)$ and~\eqref{eq:t factor} uses the right-invariance of the Haar measure $\mu$ . By combining~\eqref{eq:middle identity} and~\eqref{eq:t factor} the middle term in the right hand side of~\eqref{eq:three term semigroup} is bounded as follows.
\begin{align}
&\nonumber \left(\int_0^\infty \left(\int_{\H(\R)}\|P_t*f(hc^t)-P_t*f(x)\|_X^pd\mu(h)\right)^{q/p}\frac{dt}{t^{1+q/2}}\right)^{1/q}\\
&\nonumber \le \left(\int_0^\infty t^{\frac{q}{2}-1} \left(\int_t^\infty \left\|R_{s/2}*(Q_{s/2}*f)\right\|_{L_p(\H(\R),X)}ds\right)^qdt\right)^{1/q}\\
& \le\label{eq:for young} 2 \left(\int_0^\infty t^{\frac{3q}{2}-1} \left\|R_{t/2}*(Q_{t/2}*f)\right\|_{L_p(\H(\R),X)}^qdt\right)^{1/q},
\end{align}
where~\eqref{eq:for young} uses the second form of Hardy's inequality (see~\cite[Sec.~A.4]{Ste70}). By Young's inequality, for every $t\in (0,\infty)$ we have
\begin{multline}\label{eq:young}
\| R_{t/2}*Q_{t/2}*f\|_{L_p(\H(\R),X)}\le \|R_{t/2}\|_{L_1(\R)}\cdot \|Q_{t/2}*f\|_{L_p(\H(\R),X)}\\\stackrel{\eqref{eq:norm R}}{=} \frac{4}{\pi t}\cdot \|Q_{t/2}*f\|_{L_p(\H(\R),X)}.
\end{multline}
A substitution of~\eqref{eq:young} into~\eqref{eq:for young} yields the estimate
\begin{multline}\label{eq;change of variable}
\left(\int_0^\infty \left(\int_{\H(\R)}\|P_t*f(hc^t)-P_t*f(h)\|_X^pd\mu(h)\right)^{q/p}\frac{dt}{t^{1+q/2}}\right)^{1/q}\\\le \frac{8\sqrt{2}}{\pi}\left( \int_0^\infty  t^{\frac{q}{2}-1}  \|Q_{t}*f\|_{L_p(\H(\R),X)}^q dt\right)^{1/q}.
\end{multline}
The desired bound~\eqref{eq:g function besov} now follows by applying the triangle inequality in $L_q(t^{-1-q/2}dxdt,L_p(\H(\R),X))$ to~\eqref{eq:three term semigroup} while using \eqref{eq:first of three}, \eqref{eq:third of three}, \eqref{eq;change of variable}.
\end{proof}

\begin{proof}[Proof of Proposition~\ref{prop:main}]
Substitute Lemma~\ref{lem:integrated increment} into Lemma~\ref{lem:enter g function}.
\end{proof}

\section{Proof of Theorem~\ref{thm:main}}\label{sec:proof main}

Our goal here is to prove the following theorem, the case $p=q$ of which is Theorem~\ref{thm:main}.
\begin{theorem}\label{thm:pq-local}
For every $\eta\in (0,1)$, $q\in [2,\infty)$ and $p\in (1,q]$ there exists $K=K(\eta,p,q)\in (0,\infty)$ with the following property. Suppose that $(X,\|\cdot\|_X)$ is a Banach space satisfying $\d_{(X,\|\cdot\|_X)}(\e)\ge (\eta\e)^q$ for every $\e\in (0,1)$. Then for every $n\in \N$ and every $f:\H\to X$ we have
\begin{multline}\label{eq:desired local pq}
\left(\sum_{k=1}^{n^2}\frac{1}{k^{1+q/2}}\left(\sum_{x\in B_n} \|f(xc^k)-f(x)\|_X^p\right)^{q/p}\right)^{1/q}\\\le
K\left(\sum_{x\in B_{21n}} \Big(\|f(xa)-f(x)\|^p_X+\|f(xb)-f(x)\|^p_X\Big)\right)^{1/p}.
\end{multline}
\end{theorem}

Theorem~\ref{thm:pq-local} will be shown to follow from its continuous counterpart, i.e., Theorem~\ref{thm:real}. This deduction resembles the corresponding argument appearing in Section 7 of~\cite{ANT10}, though there are some crucial differences. Other than dealing with an inequality of a different form, the main differences follow from the fact that~\cite{ANT10} argues about a different notion of horizontal gradient, namely a coarse counterpart of $\nabla_\H$, and from the fact that~\cite{ANT10} uses a Euclidean version of Kleiner's local Poincar\'e inequality~\cite{Kle10}, while we need a similar statement for functions taking values in general Banach spaces.

\subsection{A local Poincar\'e inequality on $\H$}
Lemma~\ref{lem:metric bruce} below extends the local Poincar\'e inequality of Kleiner~\cite[Thm.~2.2]{Kle10}, who proved the same statement for real-valued functions and $p=2$. The simple argument below shows that the same result holds true for functions taking value in a general metric space, a fact that will be useful for us later. In~\cite{Kle10} the corresponding result is stated for arbitrary finitely generated groups, and while we state it for $\H$, the same argument works mutatis mutandis in the full generality of~\cite{Kle10}.

\begin{lemma}\label{lem:metric bruce} Fix $p\in [1,\infty)$ and $n\in \N$. Let $(M,d_M)$ be a metric space. For every $f:\H\to M$,
\begin{multline}\label{eq:bruce}
\sum_{x,y\in B_n} d_M(f(x),f(y))^p\\\lesssim (2n)^{p+4} \sum_{x\in B_{3n}} \Big(d_M(f(xa),f(x))^p+d_M(f(xb),f(x))^p\Big).
\end{multline}
\end{lemma}
\begin{proof}
For every $z\in B_{2n}$ choose $s_1(z),\ldots,s_{2n}(z)\in \{a,b,a^{-1},b^{-1},e_\H\}$ such that $z=s_1(z)\cdots s_{2n}(z)$. For  $i\in \{1,\ldots,2n\}$ write  $w_i(z)=s_1(z)\cdots s_i(z)$ and set $w_0(z)=e_\H$. By the triangle inequality and H\"older's inequality, for every $x,y\in B_n$ we have
\begin{multline*}
d_M(f(x),f(y))^p\\\le (2n)^{p-1}\sum_{i=0}^{2n-1} d_M\left(f\left(xw_{i}\left(x^{-1}y\right)\right),f\left(xw_{i}\left(x^{-1}y\right)s_{i+1}\left(x^{-1}y\right)\right)\right)^p.
\end{multline*}
Consequently,
\begin{align*}
&\sum_{x,y\in B_n} d_M(f(x),f(y))^p\\&\le (2n)^{p-1}\sum_{z\in B_{2n}}\sum_{i=1}^{2n-1}\sum_{x\in B_n}d_M\left(f\left(xw_{i}\left(z\right)\right),f\left(xw_{i}\left(z\right)s_{i+1}\left(z\right)\right)\right)^p\\
&= (2n)^{p-1} \sum_{z\in B_{2n}}\sum_{i=0}^{2n-1}\sum_{g\in B_nw_i(z)} d_M\left(f(g),f(gs_{i+1}(z))\right)^p\\
&\le (2n)^{p-1} \cdot |B_{2n}|\cdot 2n\sum_{h\in B_{3n}} \Big(d_M(f(ha),f(h))^p+d_M(f(hb),f(h))^p\Big).
\end{align*}
Since $|B_{2n}|\asymp n^4$, the desired inequality~\eqref{eq:bruce} follows.
\end{proof}

\subsection{Localization of a discrete global inequality}

Theorem~\ref{thm:pq-local} is a local vertical versus horizontal Poincar\'e inequality in the sense that it involves sums over balls in $\H$. While this form of the inequality is important for the deduction of lower bounds on bi-Lipschitz distortion of balls in $\H$, the natural discrete analogue of Theorem~\ref{thm:real} is the following global vertical versus horizontal Poincar\'e inequality on $\H$.

\begin{theorem}\label{thm:pq-global}
For every $\eta\in (0,\infty)$, $q\in [2,\infty)$ and $p\in (1,q]$ there exists $K=K(\eta,p,q)\in (0,\infty)$ with the following property. Suppose that $(X,\|\cdot\|_X)$ is a Banach space satisfying $\d_{(X,\|\cdot\|_X)}(\e)\ge (\eta\e)^q$ for every $\e\in (0,1)$. Then for every finitely supported $f:\H\to X$ we have
\begin{multline}\label{eq:global main}
\left(\sum_{k=1}^{\infty}\frac{1}{k^{1+q/2}}\left(\sum_{x\in \H} \|f(xc^k)-f(x)\|_X^p\right)^{q/p}\right)^{1/q}\\\le
K\left(\sum_{x\in \H} \Big(\|f(xa)-f(x)\|^p_X+\|f(xb)-f(x)\|^p_X\Big)\right)^{1/p}.
\end{multline}
\end{theorem}

Theorem~\ref{thm:pq-global} will be deduced from Theorem~\ref{thm:real} via  a partition of unity argument. This is done is Section~\ref{sec:discretization} below. We shall now assume the validity of Theorem~\ref{thm:pq-global} and proceed, using Lemma~\ref{lem:metric bruce}, to conclude the proof of Theorem~\ref{thm:pq-local}.

\begin{proof}[Proof of Theorem~\ref{thm:pq-local}]
The argument follows the proof of Claim~7.2 in~\cite{ANT10}. Fix $n\in \N$ and a finitely supported $f:\H\to X$. By translating $f$ we may assume that $\sum_{x\in B_{7n}}f(x)=0$. Due to Lemma~\ref{lem:metric bruce}, this implies that
\begin{align}\label{eq:use bruce}
&\nonumber\left(\sum_{x\in B_{7n}}\|f(x)\|_X^p\right)^{1/p}\\\nonumber&= \left(\sum_{x\in B_{7n}}\left\|\frac{1}{|B_{7n}|}\sum_{y\in B_{7n}}(f(x)-f(y))\right\|_X^p\right)^{1/p}\\
&\le \nonumber \left(\frac{1}{|B_{7n}|}\sum_{x,y\in B_{7n}}\|f(x)-f(y)\|_X^p\right)^{1/p}\\
&\lesssim n\left(\sum_{x\in B_{21n}}\Big(\|f(xa)-f(x)\|_X^p+\|f(xb)-f(x)\|_X^p\Big)\right)^{1/p}.
\end{align}

Define a cutoff function $\xi:\H\to [0,1]$ by
$$
\xi(x)\eqdef \left\{\begin{array}{ll}1& x\in B_{5n},\\
6-\frac{d_W(x,e)}{n}& x\in B_{6n}\setminus B_{5n},\\
0& x\in \H\setminus B_{6n},
\end{array}\right.
$$
and let $\phi \eqdef \xi f$. Then $\phi$ is supported on $B_{6n}$. Since $\xi$ is $\frac{1}{n}$-Lipschitz and takes values in $[0,1]$, for all $s\in \{a,b\}$ and $x\in \H$ we have
\begin{multline}\label{eq:tildef}
\|\phi(x)-\phi(xs)\|_X\le |\xi(x)-\xi(xs)|\cdot\|f(x)\|_X+|\xi(xs)|\cdot \|f(x)-f(xs)\|_X\\\le
\frac{1}{n}\|f(x)\|_X+\|f(x)-f(xs)\|_X.
\end{multline}

If $k\in \{1,\ldots, n^2\}$ then $d_W(e_\H,c^k)\le 4n$ (see~\cite{Bla03}). Consequently, for every $x\in B_n$ we have $xc^k\in B_{5n}$, and therefore $\phi(x)=f(x)$ and $\phi(xc^k)=f(xc^k)$. Hence,
\begin{multline}\label{eq:pass to phi}
\left(\sum_{k=1}^{n^2}\frac{1}{k^{1+q/2}}\left(\sum_{x\in B_n}\|f(xc^k)-f(x)\|_X^p\right)^{q/p}\right)^{1/q}\\
\le  \left(\sum_{k=1}^{\infty}\frac{1}{k^{1+q/2}}\left(\sum_{x\in \H} \|\phi(xc^k)-\phi(x)\|_X^p\right)^{q/p}\right)^{1/q}.
\end{multline}
Moreover,
\begin{align}
&\nonumber\left(\sum_{x\in \H} \Big(\|\phi(xa)-\phi(x)\|^p_X+\|\phi(xb)-\phi(x)\|^p_X\Big)\right)^{1/p}\\
&=\label{eq:use support} \left(\sum_{x\in B_{7n}} \Big(\|\phi(xa)-\phi(x)\|^p_X+\|\phi(xb)-\phi(x)\|^p_X\Big)\right)^{1/p}\\
&\le \nonumber \frac{2^{1/p}}{n}\left(\sum_{x\in B_{7n}} \|f(x)\|_X^p\right)^{1/p}\\&\qquad+\left(\sum_{x\in B_{7n}} \Big(\|f(xa)-f(x)\|^p_X+\|f(xb)-f(x)\|^p_X\Big)\right)^{1/p}\label{eq:use truncation}\\
&\lesssim \label{eq:use 21}\left(\sum_{x\in B_{21n}} \Big(\|f(xa)-f(x)\|^p_X+\|f(xb)-f(x)\|^p_X\Big)\right)^{1/p},
\end{align}
where~\eqref{eq:use support} holds true since $\phi$ is supported on $B_{6n}$, \eqref{eq:use truncation} uses~\eqref{eq:tildef}, and~\eqref{eq:use 21} uses~\eqref{eq:use bruce}. The desired inequality~\eqref{eq:desired local pq} now follows from an application of Theorem~\ref{thm:pq-global} to $\phi$, combined with~\eqref{eq:pass to phi} and~\eqref{eq:use 21}.
\end{proof}

\subsection{Discretization of Theorem~\ref{thm:real}}\label{sec:discretization} Here we prove Theorem~\ref{thm:pq-global}, thus completing the proof of Theorem~\ref{thm:pq-local}, and, as a special case, completing the proof of Theorem~\ref{thm:main}. The argument below is a variant of the proof of Claim~7.3 in~\cite{ANT10}.

We will use the following simple lemma, whose proof is similar to the proof of Lemma~\ref{lem:metric bruce}. Below, a mapping $f:\H\to M$ is said to be finitely supported if there exists $m_0\in M$ such that $|f^{-1}(M\setminus\{m_0\})|<\infty$.

\begin{lemma}\label{lem:global poincare-simple} Fix $p\in [1,\infty)$ and $n\in \N$. Let $(M,d_M)$ be a metric space. For every finitely supported $f:\H\to M$,
\begin{multline}\label{eq:bruce}
\sum_{y\in \H}\sum_{z\in B_n} d_M(f(yz),f(y))^p\\\lesssim n^{p+4} \sum_{x\in \H} \Big(d_M(f(xa),f(x))^p+d_M(f(xb),f(x))^p\Big).
\end{multline}
\end{lemma}

\begin{proof}
For every $z\in B_{n}$ choose $s_1(z),\ldots,s_{n}(z)\in \{a,b,a^{-1},b^{-1},e_\H\}$ such that $z=s_1(z)\cdots s_{n}(z)$. Set $w_0(z)=e_\H$ and for  $i\in \{1,\ldots,n\}$ write  $w_i(z)=s_1(z)\cdots s_i(z)$. By the triangle inequality and H\"older's inequality, for every $y\in \H$ we have
\begin{equation*}
d_M(f(yz),f(y))^p\le n^{p-1}\sum_{i=0}^{n-1} d_M\left(f\left(yw_{i}\left(z\right)\right),f\left(yw_{i}\left(z\right)s_{i+1}\left(z\right)\right)\right)^p.
\end{equation*}
Consequently,
\begin{multline*}
\sum_{y\in \H}\sum_{z\in B_n} d_M(f(yz),f(y))^p\\
\le n^p|B_n|\sum_{x\in \H} \Big(d_M(f(xa),f(x))^p+d_M(f(xb),f(x))^p\Big).\tag*{\qedhere}
\end{multline*}
\end{proof}

\begin{proof}[Proof of Theorem~\ref{thm:pq-global}] Since $\H$ is a co-compact lattice in $\H(\R)$, there exists a compactly supported smooth function $\chi:\H(\R)\to [0,1]$ with
\begin{equation}\label{eq:partition of unity}
\forall\, h\in \H(\R),\qquad \sum_{x\in \H} \chi_x(h)=1,
\end{equation}
where $\chi_x:\H(\R)\to X$ is given by $\chi_x(h)=\chi(x^{-1}h)$ for every $x\in \H$ and $h\in \H(\R)$. Let $A\subseteq \H(\R)$ denote the support of $\chi$. We may assume without loss of generality that $A^{-1}=A$. Note that due to~\eqref{eq:partition of unity} we have
\begin{equation}\label{eq:A covers}
\bigcup_{x\in \H}xA=\H(\R).
\end{equation}
 Since $A$ is compact, we may fix  $m\in \N$ for which $A\cap \H\subseteq B_m$.

Let $f:\H\to X$ be finitely supported. Define $F:\H(\R)\to X$ by
\begin{equation}\label{eq:defF}
F(h)\eqdef \sum_{x\in \H} \chi_x(h)f(x).
\end{equation}
Fix $y\in \H$ and $h\in yA$. Note that it follows from~\eqref{eq:partition of unity} that

\begin{equation}\label{eq:gradient sum}
\sum_{x\in \H} \nabla_\H\chi_x(h)=0.
\end{equation}
Observe that if $x\in \H$ satisfies $\nabla_\H\chi_x(h)\neq 0$ then necessarily $x^{-1}h\in A$. Since $A^{-1}=A$, this implies that $x\in hA\subseteq (yA)A\subseteq yB_{2m}$. Hence,
\begin{align}\label{eq:pointwise on yA}
\nonumber&\left\|\nabla_\H F(h)\right\|_{\ell_p^2(X)}\\
\nonumber&=\left\|\sum_{x\in yB_{2m}}\big(\partial_a\chi_x(h)(f(x)-f(y)),\partial_b\chi_x(h)(f(x)-f(y))\big)\right\|_{\ell_p^2(X)}
\\\nonumber&\le \left(\max_{h\in A} \left\|\nabla_\H\chi(h)\right\|_{\ell_p^2}\right)\sum_{z\in B_{2m}}\|f(yz)-f(y)\|_X\\
&\lesssim |B_{2m}|^{1-1/p}\left(\sum_{z\in B_{2m}}\|f(yz)-f(y)\|_X^p\right)^{1/p}.
\end{align}
By integrating~\eqref{eq:pointwise on yA} over $yA$, we have
\begin{equation}\label{eq:integral over yA}
\int_{yA} \left\|\nabla_\H F(h)\right\|_{\ell_p^2(X)}^pd\mu(h)\lesssim \sum_{z\in B_{2m}}\|f(yz)-f(y)\|_X^p.
\end{equation}
By summing~\eqref{eq:integral over yA} over $y\in \H$ and recalling~\eqref{eq:A covers}, we conclude that
\begin{multline}\label{gradient control}
\left(\int_{\H(\R)} \left\|\nabla_\H F(h)\right\|_{\ell_p^2(X)}^pd\mu(h)\right)^{1/p}\lesssim \left(\sum_{y\in \H} \sum_{z\in B_{2m}}\|f(yz)-f(y)\|_X^p\right)^{1/p}\\
\lesssim \left(\sum_{x\in \H} \Big(\|f(xa)-f(x)\|_X^p+\|f(xb)-f(x)\|_X^p\Big)\right)^{1/p},
\end{multline}
where the final step of~\eqref{gradient control} uses Lemma~\ref{lem:global poincare-simple}.

Next, let $U\subseteq \H(\R)$ be a bounded open set such that $e_\H\in U$ yet $U\cap (xU)=\emptyset$ for all $x\in \H\setminus \{e_\H\}$.  Since $c^{[0,1]}=\{c^s:s\in [0,1]\}$ and $U$ are bounded subsets of $\H(\R)$, we can choose $r\in \N$ such that $(Uc^{[0,1]}A)\cap \H\subseteq B_r$. Fix $x\in \H$ and $h\in xU$. Suppose that $k\in \N$ and $t\in [k,k+1]$. Recalling~\eqref{eq:partition of unity} and~\eqref{eq:defF} we have
\begin{equation}\label{eq:error with center}
F(hc^t)-f(xc^k)=\sum_{w\in \H} \chi_w(hc^t)(f(w)-f(xc^k)).
\end{equation}
Observe that if $w\in \H$ satisfies $\chi_w(hc^t)\neq 0$ then $hc^t\in A$, and since $A=A^{-1}$ and $c$ belongs to the center of $\H(\R)$, this inclusion implies that
$w\in hc^tA\subseteq xc^k Uc^{[0,1]}A\subseteq xc^kB_r$. \eqref{eq:error with center} therefore implies that

\begin{equation}\label{eq:ftoF1}
\|F(hc^t)-f(xc^k)\|_X^p\lesssim \sum_{z\in B_r} \|f(xc^kz)-f(xc^k)\|_X^p.
\end{equation}
Similar (simpler) reasoning shows that also
\begin{equation}\label{eq:ftoF2}
\|F(h)-f(x)\|_X^p\lesssim \sum_{z\in B_r} \|f(xz)-f(x)\|_X^p.
\end{equation}
It follows from~\eqref{eq:ftoF1} and~\eqref{eq:ftoF2} that
\begin{multline}\label{eq:for integration over xU}
\|f(xc^k)-f(x)\|_X^p\lesssim \|F(hc^t)-F(h)\|_X^p\\+\sum_{z\in B_r}\Big(\|f(xc^kz)-f(xc^k)\|_X^p+\|f(xz)-f(x)\|_X^p\Big).
\end{multline}
Integration of~\eqref{eq:for integration over xU} over $h\in xU$ therefore yields the estimate
 \begin{multline}\label{eq:integrated over xU}
\|f(xc^k)-f(x)\|_X^p\lesssim \int_{xU}\|F(hc^t)-F(h)\|_X^pd\mu(h)\\+\sum_{z\in B_r}\Big(\|f(xc^kz)-f(xc^k)\|_X^p+\|f(xz)-f(x)\|_X^p\Big).
\end{multline}
Since the sets $\{xU\}_{x\in \H}$ are pairwise disjoint, summation of~\eqref{eq:integrated over xU} over $x\in \H$, combined with an application of Lemma~\ref{lem:global poincare-simple}, shows that
\begin{multline}\label{eq:before t integration}
\left(\sum_{x\in \H}\|f(xc^k)-f(x)\|_X^p\right)^{1/p}\lesssim \left(\int_{\H(\R)}\|F(hc^t)-F(h)\|_X^pd\mu(h)\right)^{1/p}\\+\left(\sum_{x\in \H} \Big(\|f(xa)-f(x)\|_X^p+\|f(xb)-f(x)\|_X^p\Big)\right)^{1/p}.
\end{multline}
Integration of~\eqref{eq:before t integration} over $t\in [k,k+1]$ yields
\begin{align}\label{eq:t integrated}
&\nonumber\frac{1}{k^{1+q/2}}\left(\sum_{x\in \H}\|f(xc^k)-f(x)\|_X^p\right)^{q/p}\\&\le \nonumber C^q\int_k^{k+1}\left(\int_{\H(\R)}\|F(hc^t)-F(h)\|_X^pd\mu(h)\right)^{q/p}\frac{dt}{t^{1+q/2}}\\&\quad+\frac{C^q}{k^{1+q/2}}\left(\sum_{x\in \H} \Big(\|f(xa)-f(x)\|_X^p+\|f(xb)-f(x)\|_X^p\Big)\right)^{q/p},
\end{align}
where $C\in (0,\infty)$ is a universal constant. We may now sum~\eqref{eq:t integrated} over $k\in \N$ to get the bound

\begin{align}\label{eq:vertical control}
&\nonumber\left(\sum_{k=1}^\infty\frac{1}{k^{1+q/2}}\left(\sum_{x\in \H}\|f(xc^k)-f(x)\|_X^p\right)^{q/p}\right)^{1/q}\\&\lesssim \nonumber\left(\int_0^{\infty}\left(\int_{\H(\R)}\|F(hc^t)-F(h)\|_X^pd\mu(h)\right)^{q/p}\frac{dt}{t^{1+q/2}}\right)^{1/q}\\&\quad+
\left(\sum_{x\in \H} \Big(\|f(xa)-f(x)\|_X^p+\|f(xb)-f(x)\|_X^p\Big)\right)^{1/p}.
\end{align}
The desired inequality~\eqref{eq:global main} follows from an application of Theorem~\ref{thm:real} to $F$, and substituting~\eqref{gradient control} and~\eqref{eq:vertical control} into the resulting inequality.
\end{proof}

\section{Vertical perimeter versus horizontal perimeter}\label{sec:conj}
The case $X=\R$ and $q=2$ of Theorem~\ref{thm:real} shows that for every $p\in (1,2]$ and every smooth and compactly supported $f:\H(\R)\to \R$,
\begin{multline}\label{eq:toR}
\left(\int_0^\infty \left(\int_{\H(\R)}|f(hc^t)-f(h)|^pd\mu(h)\right)^{2/p}\frac{dt}{t^{2}}\right)^{1/2}\\\lesssim_{p} \left(\int_{\H(\R)} \left\|\nabla_\H f(h)\right\|_{\ell_p^2}^p d\mu(h)\right)^{1/p}.
\end{multline}
The implied constant in~\eqref{eq:toR} that follows from our proof of Theorem~\ref{thm:real} tends to $\infty$ as $p\to 1$. However, we ask whether the endpoint case $p=1$ of~\eqref{eq:toR} does nevertheless hold true.
\begin{question}\label{Q:strongest}
Is it true that every smooth and compactly supported $f:\H(\R)\to \R$  satisfies
\begin{multline}\label{eq:p=1}
\left(\int_0^\infty \left(\int_{\H(\R)}|f(hc^t)-f(h)|d\mu(h)\right)^{2}\frac{dt}{t^{2}}\right)^{1/2}\\\lesssim \int_{\H(\R)} \left\|\nabla_\H f(h)\right\|_{\ell_1^2} d\mu(h).
\end{multline}
\end{question}
A standard application of the co-area formula shows that it suffices to prove~\eqref{eq:p=1} when $f$ is an indicator of a measurable set $A\subseteq \H(\R)$. For such a choice of $f$ the right hand side of~\eqref{eq:p=1} should be interpreted as the horizontal perimeter of $A$, denoted $\mathrm{PER}(A)$. Rather than defining the horizontal perimeter $\mathrm{PER}(A)$ here, we refer to~\cite{Amb01} and~\cite[Sec.~2]{CK10} for a detailed discussion of this notion.

\begin{definition}[Vertical perimeter at scale $t$]\label{def vert}
Let $A\subseteq \H(\R)$ be measurable and $t\in (0,\infty)$. Recalling that $\mu$ is the Haar measure on $\H(\R)$ (equivalently $\mu$ is the Lebesgue measure on $\R^3$), define the vertical perimeter of $A$ at scale $t$, denoted $v_t(A)$, to be the quantity
\begin{equation}\label{eq:def Vt}
v_t(A)\eqdef \mu\left(\left\{h\in A:\ hc^t\notin A\quad\mathrm{or}\quad hc^{-t}\notin A\right\}\right).
\end{equation}
\end{definition}
Thus $v_t(A)$ measures the the size of those points of $A$ from which a vertical movement of $\pm t$ lands outside $A$. Using this terminology, we  have the following reformulation of Question~\ref{Q:strongest} in terms of an isoperimetric-type inequality.
\begin{question}\label{Q:isoperimetric}
Is it true that for every measurable $A\subseteq \H(\R)$ one has
\begin{equation}\label{eq:con isoperimetric}
\int_0^\infty \frac{v_t(A)^2}{t^2}dt\lesssim \mathrm{PER}(A)^2.
\end{equation}
\end{question}
While we believe that Question~\ref{Q:isoperimetric} has a positive answer, at present we do not have sufficient evidence that would justify formulating this assertion as a conjecture. However, there is a weaker coarse variant of Question~\ref{Q:isoperimetric} for which there is significant partial positive evidence that will be published elsewhere; this evidence originates from ongoing work (including numerical experiments and proofs of nontrivial special cases) on our question by Artem Kozhevnikov and Pierre Pansu (personal communication). We shall now formulate this weaker conjecture, and proceed to explain an  application of it to theoretical computer science.

\begin{definition}[Coarse total vertical perimeter at resolution $\e$] Let $A\subseteq \H(\R)$ be measurable and $\e\in (0,1)$. Define the coarse total vertical perimeter of $A$ at resolution $\e$ by
\begin{equation}\label{def coarse}
V^{(\e)}(A)\eqdef \int_\e^1 \frac{v_t(A)}{t^{3/2}}dt.
\end{equation}
\end{definition}
The following isoperimetric-type conjecture relates  the coarse total vertical perimeter of $A$ at resolution $\e$ with its horizontal perimeter.
\begin{conjecture}\label{total conjecture}
For every measurable $A\subseteq \H(\R)$ and every $\e\in (0,1/2)$,
\begin{equation}\label{eq:total conjecture}
V^{(\e)}(A)\lesssim \sqrt{\log(1/\e)}\cdot \mathrm{PER}(A).
\end{equation}
\end{conjecture}
Note that if~\eqref{eq:con isoperimetric} holds true then by the Cauchy--Schwarz inequality,
\begin{multline}\label{eq:CS}
V^{(\e)}(A)\stackrel{\eqref{def coarse}}{=}\int_\e^1 \frac{1}{\sqrt{t}}\cdot \frac{v_t(A)}{t}dt\\\le \sqrt{\int_\e^1\frac{dt}{t}}\cdot \sqrt{\int_\e^1 \frac{v_t(A)^2}{t^2}dt}\stackrel{\eqref{eq:con isoperimetric}}{\lesssim} \sqrt{\log(1/\e)}\cdot \mathrm{PER}(A).
\end{multline}
This shows that a positive answer to Question~\ref{Q:isoperimetric} would imply a positive resolution of Conjecture~\ref{total conjecture}.

By the co-area formula, Conjecture~\ref{total conjecture} has the following equivalent functional version: for every smooth and compactly supported $f:\H(\R)\to \R$ and every $\e\in (0,1/2)$,
$$
\int_\e^1\int_{\H(\R)}\frac{|f(hc^t)-f(h)|}{t^{3/2}}d\mu(h) dt \lesssim \sqrt{\log (1/\e)}\int_{\H(\R)}\|\nabla_\H f(h)\|_{\ell_1^2} d\mu(h).
$$
By re-scaling, this is equivalent to requiring that for every $R\in (1,\infty)$,
\begin{multline}\label{eq:continuous log R}
\int_1^R\int_{\H(\R)}\frac{|f(hc^t)-f(h)|}{t^{3/2}}d\mu(h) dt\\ \lesssim \sqrt{\log R}\int_{\H(\R)}\|\nabla_\H f(h)\|_{\ell_1^2} d\mu(h).
\end{multline}
Arguing as in Section~\ref{sec:proof main}, one sees that~\eqref{eq:continuous log R} implies that for every integer $n\ge 2$ and every $f:\H\to \R$ we have
\begin{multline}\label{eq:log conjecture discrete}
\sum_{k=1}^{n^2}\sum_{x\in B_n}\frac{ |f(xc^k)-f(x)|}{k^{3/2}}\\\lesssim
\sqrt{\log n}\sum_{x\in B_{21n}} \Big(|f(xa)-f(x)|+|f(xb)-f(x)|\Big).
\end{multline}
By summing~\eqref{eq:log conjecture discrete} over coordinates we conclude that the conjectural inequality~\eqref{eq:log conjecture discrete} implies that every $f:\H\to \ell_1$ satisfies
\begin{multline}\label{eq:conjecture ell1 version}
\sum_{k=1}^{n^2}\sum_{x\in B_n}\frac{ \|f(xc^k)-f(x)\|_1}{k^{3/2}}\\\lesssim
\sqrt{\log n}\sum_{x\in B_{21n}} \Big(\|f(xa)-f(x)\|_1+\|f(xb)-f(x)\|_1\Big).
\end{multline}
The same computation as in~\eqref{eq:distortion computation} shows that if~\eqref{eq:conjecture ell1 version} does indeed hold true, then $c_1(B_n,d_W)\gtrsim \sqrt{\log n}$, and therefore, as explained in the introduction, in fact  $c_1(B_n,d_W)\asymp \sqrt{\log n}$. We state this conclusion as a separate conjecture.

\begin{conjecture}\label{conj:l1 distortion}
For every integer $n\ge 2$ we have $c_1(B_n,d_W)\asymp \sqrt{\log n}$.
\end{conjecture}

 To summarize the above reasoning, a positive answer to Question~\ref{Q:strongest} (or equivalently Question~\ref{Q:isoperimetric}) implies the positive resolution of Conjecture~\ref{total conjecture}, which in turn implies the positive resolution of Conjecture~\ref{conj:l1 distortion}.

We recall that in~\cite{ckn} it was shown that $c_1(B_n,d_W)\gtrsim (\log n)^\kappa$ for some universal constant $\kappa>0$. While the proof of this result in~\cite{ckn} is constructive, to the best of our knowledge there was no serious attempt to use that proof in order to give a good estimate on the value of $\kappa$, since it seems very unlikely that the methods of~\cite{ckn} can yield the sharp bound $\kappa\ge \frac12$. Suppose that for some $p\in [1,\infty)$ the following variant of~\eqref{eq:con isoperimetric} holds true for every measurable $A\subseteq \H(\R)$.
\begin{equation}\label{eq:p variant}
\left(\int_0^\infty \frac{v_t(A)^p}{t^{1+p/2}}dt\right)^{1/p}\lesssim \mathrm{PER}(A).
\end{equation}
Arguing as in~\eqref{eq:CS}, an application of H\"older's inequality would yield the validity of~\eqref{eq:total conjecture} with the term $\sqrt{\log(1/\e)}$ replaced by $(\log(1/\e))^{1-1/p}$. Reasoning identically to the discussion preceding Conjecture~\ref{conj:l1 distortion}, it would follow that $c_1(B_n,d_W)\gtrsim (\log n)^{1/p}$. Thus, proving an inequality such as~\eqref{eq:p variant} would yield a new proof of the lower bound on $c_1(B_n,d_W)$ of~\cite{ckn}. It is straightforward to check that~\eqref{eq:p variant} holds true when $p=\infty$, and it would be interesting to investigate whether the method of~\cite{ckn} can potentially lead to a proof of~\eqref{eq:p variant} for some finite $p$.

 Due to the fact that our conjectures lead to an asymptotically sharp lower bound on the $\ell_1$ distortion of balls in $\H$, it follows that the smallest possible $p\in [1,\infty)$ for which~\eqref{eq:p variant} can hold true is $p=2$. Correspondingly, \eqref{eq:total conjecture} cannot be improved. The three known embeddings of $(B_n,d_W)$ into $\ell_1$ with distortion $O(\sqrt{\log n})$ that were mentioned in the introduction (first embed into $\ell_2$, and then use the fact that finite subsets of $\ell_2$ embed isometrically into $\ell_1$)  lead to examples of sets $A\subseteq \H(\R)$ for which~\eqref{eq:total conjecture} is sharp. A different explicit example of a set for which~\eqref{eq:total conjecture} is sharp was found by Robert Young (personal communication). This illuminating example will be published elsewhere.

 \subsection{Algorithmic implications}
 We shall now indicate an important consequence of the conjectures presented here to theoretical computer science. This topic is discussed at length in e.g.~\cite{ckn,Naor10}, so we will be  brief here, and in particular we will mention only the best known results without describing the historical development.

 In the Sparsest Cut problem one is given as input two symmetric functions $C,D:\{1,\ldots,n\}\times \{1,\ldots,n\}\to [0,\infty)$ and the goal is to compute (or estimate) in polynomial time the quantity

 \begin{equation*}
\Phi^*(C,D)\eqdef \min_{\emptyset \neq S\subsetneq \{1,\ldots,n\}}
\frac{\sum_{i=1}^n\sum_{j=1}^nC(i,j)\left|\1_S(i)-\1_S(j)\right|}{\sum_{i=1}^n\sum_{j=1}^nD(i,j)\left|\1_S(i)-\1_S(j)\right|}\,  .
\end{equation*}
 This versatile optimization problem is of central algorithmic importance; see e.g.~\cite{Shmoys95} for examples of its applicability.

  The best known approximation algorithm for the Sparsest Cut problem was proposed by Goemans and Linial~\cite{Goe97,Lin02}. In~\cite{ALN08} it is proved that the Goemans-Linial algorithm outputs a number which is guaranteed to be within a factor of $O(\sqrt{\log n}\log\log n)$ of $\Phi^*(C,D)$. Goemans~\cite{Goe97} and Linial~\cite{Lin02} (see also~\cite[pages~379--380]{Mat-Discrete-Geometry}) actually conjectured  that their algorithm outputs a number which is within a factor of $O(1)$ of $\Phi^*(C,D)$. The Goemans-Linial conjecture was disproved by Khot and Vishnoi~\cite{KV04}, who proved that the Goemans-Linial algorithm makes an error of at least a constant power of $\log\log n$ on some inputs. A link of the Sparsest Cut problem to the Heisenberg group was found in~\cite{LN06}, where the Goemans-Linial algorithm was shown to make an error of at least a constant multiple of $c_1(B_n,d_W)$ on some inputs. The lower bound $c_1(B_n,d_W)\gtrsim (\log n)^\kappa$ of~\cite{ckn} consequently established that the Goemans-Linial algorithm makes an error of at least a constant multiple of $(\log n)^\kappa$ on some inputs. An affirmative solution of Conjecture~\ref{conj:l1 distortion} would improve this lower bound to a constant multiple of $\sqrt{\log n}$, thus resolving  (up to iterated logarithms) the problem of understanding the asymptotic performance of the Goemans-Linial algorithm.

\subsection*{Acknowledgements} We are grateful to Artem Kozhevnikov, Pierre Pansu and Robert Young for enlightening discussions on the questions presented in Section~\ref{sec:conj}, and for sharing  their ongoing work on this topic.

\bibliographystyle{alphaabbrvprelim}
\bibliography{besov}

\end{document}